\documentclass[12pt]{amsart}
\usepackage[margin=1in]{geometry}
 \numberwithin{equation}{section}

\newcommand{\ets}{\mbox{$e^{3\psi} $}}
\newcommand{\hc}{\mbox{$H(\overline{\Omega}) $}}
\newcommand{\ens}{\mbox{$e^{-\psi} $}}
\newcommand{\dut}{\mbox{$\left(\Delta u\right)^{-3} $}}
\newcommand{\duo}{\mbox{$\left(\Delta u\right)^{-1} $}}
\newcommand{\duw}{\mbox{$\left(\Delta u\right)^{-2} $}}
\newcommand{\ut}{\mbox{$\left(\Delta u_0(x)\right)^{-3} $}}
\newcommand{\bt}{\mbox{$\left(b_1(x)\right)^{-3} $}}
\newcommand{\btj}{\mbox{$\left(b_{1j}(x)\right)^{-3} $}}
\newcommand{\bo}{\mbox{$b_1(x) $}}
\newcommand{\boj}{\mbox{$b_{1j}(x) $}}
\newcommand{\bzj}{\mbox{$b_{0j}(x) $}}

\newcommand{\bz}{\mbox{$b_0(x) $}}

\newcommand{\skx}{\mbox{$\psi_k (x)$}}
\newcommand{\esk}{\mbox{$e^{3\psi_k }$}}

\newcommand{\esbj}{\mbox{$e^{3\overline{\psi}_j}$}}

\newcommand{\uk}{\mbox{$u_k $}}

\newcommand{\uko}{\mbox{$u_{k-1} $}}

\newcommand{\uvt}{\mbox{$\frac{u-v}{\tau} $}}

\newcommand{\utj}{\mbox{$\tilde{u}_j $}}
\newcommand{\ctj}{\mbox{$\tilde{\sigma}_j $}}
\newcommand{\ubj}{\mbox{$\overline{u}_j $}}
\newcommand{\sbj}{\mbox{$\overline{\psi}_j $}}
\newcommand{\rbj}{\mbox{$\overline{\rho}_j $}}
\newcommand{\rk}{\mbox{$\rho_k $}}
\newcommand{\rtj}{\mbox{$\tilde{\rho}_j $}}
\newcommand{\omt}{\mbox{$\Omega_T$}}
\newcommand{\io}{\mbox{$\int_\Omega$}}
\newcommand{\iot}{\mbox{$\int_{\Omega_T}$}}

\newcommand{\te}{\mbox{$\theta_\varepsilon $}}

\newtheorem{prop}{Proposition}[section]

\newtheorem{thm}{Theorem}[section]
\newtheorem{lem}{Lemma}[section]

\begin{document}
	\title[Analytic validation of a continuum model]{Analytical validation of a continuum model for the evolution of a crystal surface in multiple space dimensions }
	\author{Jian-Guo Liu and Xiangsheng Xu}\thanks
	{Liu's address: Department of Physics and Department of Mathematics, Duke University, Durham, NC 27708. {\it Email:} jliu@phy.duke.edu.\\
	\indent Xu's address:	Department of Mathematics and Statistics, Mississippi State
		University, Mississippi State, MS 39762.
		{\it Email}: xxu@math.msstate.edu. To appear in SIAM J. Math. Anal..}
 
	\begin{abstract} In this paper we are concerned with the existence of a weak solution to the initial boundary value problem for the equation $\frac{\partial u}{\partial t} = \Delta\left(\Delta u\right)^{-3}$. This problem arises in the mathematical modeling of the evolution
		of a crystal surface. Existence of a weak solution $u$ with $\Delta u\geq 0$ is obtained via a suitable substitution. Our investigations reveal the close connection between this problem and the equation $\partial_t\rho+\rho^2\Delta^2\rho^3=0$,  another crystal surface model first proposed
		by H. Al Hajj Shehadeh, R. V. Kohn and J. Weare in Physica D: Nonlinear Phenomena, {\bf 240} (2011), no. 21, 1771-1784.
	\end{abstract}

	\maketitle

	\section{Introduction}
	\subsection{Problems and physical motivation}
	Let $\Omega$ be a bounded domain in $\mathbb{R}^N$ with $C^2$ boundary $\partial\Omega$.  For $T>0$ set
	\begin{equation*}
	\omt= \Omega\times(0,T),\ \ \ \Sigma_T= \partial\Omega\times(0,T).
	\end{equation*} Consider the initial boundary value problem
		\begin{eqnarray}
		\frac{\partial u}{\partial t} &=& \Delta\dut  \ \ \ \mbox{in $\omt$,}\label{p1}\\
		u &=& b_0 (x)\ \ \ \mbox{ on $\Sigma_T$},\label{p2}\\
		\Delta u &=& b_1 (x)\ \ \ \mbox{ on $\Sigma_T$},\label{p3}\\
		u(x,0)&=& u_0(x) \ \ \ \mbox{on $\Omega$}\label{p4}
		\end{eqnarray}
			for given data $b_0(x), \,\bo$, and $u_0(x)$ with properties:
			\begin{enumerate}
			\item[(H1)] $b_0(x)\in W^{1,2}(\Omega)\cap  L^\infty(\Omega) $;
			\item[(H2)] $b_1(x)\in W^{2,2}(\Omega)\cap  L^\infty(\Omega)$ and
			there is a positive number $c_0$ such that
			\begin{equation}
			b_1(x)\geq c_0\ \ \mbox{a.e. on $\Omega$,} 
			\end{equation}
			and thus we also have that $\bt\in W^{2,2}(\Omega)$;
			\item[(H3)] $u_0(x)\in W^{2,2}(\Omega)$, $\Delta u_0(x)\geq c_0$ a.e. on $\Omega$, and $\ut\in W^{2,2}(\Omega)$.
		\end{enumerate}
		One should have the compatibility condition
		$$\bo=\Delta u_0\ \ \ \mbox{on $\partial\Omega$.}$$
		We investigate the existence of a solution $u$ to (\ref{p1})-(\ref{p4}) with $\Delta u\geq 0$.
		Our interest in this problem originated in the mathematical modeling of the evolution of a crystal surface. Such a surface below the roughing temperature consists of steps and terraces. By the Burton, Cabrera and Frank (BCF) model \cite{BCF}, atoms detach from the steps, diffuse across terraces, and reattach at new locations, inducing an overall evolution of the crystal surface. At the nanoscale, the motion of steps is described by large systems of ordinary differential equations for step positions \cite{AKW, GLL2}. At the macroscale, this description is often reduced conveniently to nonlinear PDEs for macroscopic variables, e.g., for the surface height and slope profiles (see \cite{AKW, GLL2} and the references therein). The manufacturing of crystal films lies at the heart of modern nanotechnology. How to accurately predict the motion of a crystal surface is of fundamental importance. In this effort, the development of continuum models has gained momentum. A number of PDE models have been proposed. See \cite{AKW, KDM, X, XX, MW} and the references therein. They are obtained as the continuum limit of a family of kinetic Monte Carlo models of crystal surface relaxation that includes both the solid-on-solid and discrete Gaussian models. 
		
		If the space dimension $N$ is $1$, (\ref{p1})-(\ref{p4}) has a clear connection to a couple of existing models.
		 Set
		\begin{equation}\label{rr1}
		\rho=\duo.
		\end{equation}
		Take the Laplacian of both sides of (\ref{p1})  and substitute $\Delta u=\rho^{-1}$ into the resulting equation to obtain
		\begin{equation}
		\partial_t\rho^{-1}=\Delta^2\rho^3.
		\end{equation}
		Note that $b_0(x)$ is assumed to be a function of $x$ only. Thus by (\ref{p1}), we have  that
		\begin{equation}\label{ptu}
		\Delta\rho^3=\partial_t u=0\ \ \mbox{on $\Sigma_T$.}
		\end{equation}
		Consequently, $\rho$ satisfies the problem
		\begin{eqnarray}
	\frac{\partial \rho}{\partial t}+\rho^2\Delta^2\rho^3 &=& 0  \ \ \ \mbox{in $\omt$,}\label{cr1}\\
	\rho &=&\rho_b(x) \ \ \ \mbox{ on $\Sigma_T$},\label{cr2}\\
	\Delta \rho^3 &=& 0\ \ \ \mbox{ on $\Sigma_T$},\label{cr3}\\
	\rho(x,0)&=&\rho_0(x)  \ \ \ \mbox{on $\Omega$,}\label{cr4}
	\end{eqnarray}
	where
	$$\rho_b(x)	=\frac{1}{b_1 (x)},\ \ \ \rho_0(x)= \frac{1}{\Delta u_0(x)}.$$
	 In \cite{AKW}, (\ref{cr1})-(\ref{cr4}) is proposed  as a continuum model for the evolution of an one-dimensional monotone step train separating two facets of a crystal surface. In this case, $x\in \mathbb{R}^1$ is the surface height and $\rho$ the surface slope. Since the surface height is increasing, we have that $\rho\geq 0$. This together with (\ref{rr1}) indicates
	 \begin{equation}
	 \Delta u\geq 0\ \ \ \mbox{on $\Omega$.}
	 \end{equation}

Starting with the ODE's for the velocities of the steps, the authors of \cite{AKW} analyzed the system of ODE's giving the evolution of the discrete slopes and obtained its limiting behaviors both as the time goes to infinity and as the number of steps goes to infinity, while \cite{MN} complemented the results by focusing its attention on the self-similar solutions. Questions concerning existence, uniqueness, and asymptotic self-similarity for the continuum model suggested in \cite{AKW} are left open there. 
In \cite{GLL},  the existence of 
a weak solution to (\ref{cr1}) coupled with the initial and periodic boundary conditions is established.
However, the construction of the weak solution there depends on the assumption that the space dimension is $1$ in an essential way. 
Another physical perspective from which to view (\ref{p1})- (\ref{p4}) can be obtained from  \cite{GLL}.  
Denote by $\varphi(x,t)$ the step location of a crystal surface as a function of the surface height $x$ and the time $t$. Then $\phi=\partial_x\varphi$ satisfies the equation
\begin{equation}
\partial_t\phi=\partial_{xxxx}\left(\frac{1}{\phi^3}\right).
\end{equation}
Thus the function $u$ in (\ref{p1}) is a second order anti-derivative of $\phi$, i.e., $\partial_{xx} u=\phi$.


As observed in \cite{AKW}, working with the surface slope as a function of height has its advantages, but this can only be done in one space dimension. 	In \cite{LX}, we investigate the problem
	\begin{eqnarray}
	\frac{\partial u}{\partial t} &=& \Delta e^{-\Delta u}  \ \ \ \mbox{in $\omt$,}\label{o1}\\
	\nabla u\cdot\nu=\nabla e^{-\Delta u}\cdot\nu &=& 0 \ \ \ \mbox{ on $\Sigma_T$},\label{o2}\\
	u(x,0)&=& u_0(x) \ \ \ \mbox{on $\Omega$},\label{o3}
	\end{eqnarray}
	where $\nu$ is the unit outward normal vector to the boundary, in multiple space dimensions.  This problem was first proposed in \cite{MW} to model crystal surfaces. 
	In this case, $u$ is the surface height. Obviously, (\ref{p1}) can be obtained by replacing the exponential function $e^{-s}$ in (\ref{o1}) by the power function $s^{-3}$. This connection is the key to our mathematical analysis of (\ref{p1})-(\ref{p4}). In fact, we will draw many inspirations from \cite{LX}. In spite of this, it is important to note the difference between the two problems. In \cite{LX}, we are just interested in the existence of a solution, while in (\ref{p1})-(\ref{p4}) we must obtain a solution $u$ with $\Delta u\geq 0$.
	
	As explained in \cite{LX}, the equation (\ref{o1}) can be recast in a variational form. Denote an energy functional by
	${\mathcal E}(u) = \tfrac12 \int_{\Omega} |\nabla u|^2\,dx$. 
	Equation (\ref{o1}) can be recasted as a Cahn-Hilliard equation with 
	curvature-dependent mobility
	\begin{equation} \label{variational}
	\frac{\partial u}{\partial t} 
	=  \mbox{div}\left(\mathcal{M}  \nabla \frac{\delta {\mathcal E} }{\delta u}\right), \quad 
	\mathcal{M}=e^{ \frac{\delta {\mathcal E}}{\delta u} }  
	=e^{- \Delta u }  
	\end{equation}
	and it possesses an energy-dissipation relation
	$$
	\frac{ d {\mathcal E}}{d t} 
	= -  \int_{\Omega} e^{ \frac{\delta {\mathcal E}}{\delta u} }  | \nabla \frac{\delta {\mathcal E} }{\delta u} |^2\, dx
	= -  4 \int_{\Omega} | \nabla e^{ \tfrac12\frac{\delta {\mathcal E}}{\delta u} } |^2\, dx \,.
	$$
	The exponential nonlinearity in the curvature-dependent mobility models the asymmetric behavior of the convex and concave  crystal surface in solid on solid interface growth.  Some in depth discussions on the exponential nonlinearity and numerical simulations
	were conducted in \cite{MW}. Similarly, the equation (\ref{p1}) can also be recast in the above variational form with the curvature-dependent mobility given by
	$$
	\mathcal{M}=\frac{3}{(\Delta u)^4} \,.
	$$
	Note that this curvature-dependent mobility has even more singular behavior than that of the exponential one in (\ref{variational}).
%

 The objective of this paper is to develop an existence theory for (\ref{p1})-(\ref{p4}). As a by-product of our development, we shed light on the existence assertion for (\ref{cr1})-(\ref{cr4}) in high space dimensions. In particular,
the existence result in \cite{GLL2} is a consequence of our development.

	\subsection{A priori estimates}
	To motivate the definition of a weak solution, we need to derive a priori estimates for smooth solutions of (\ref{p1})-(\ref{p4}) with $\Delta u\geq 0$ on $\Omega_T$.
	To begin, 
	we
	square both sides of the equation $\partial_tu-\Delta\rho^3=0$,  integrate the resulting one over $\Omega$, and thereby obtain
	\begin{equation}\label{e11}
	-2\int_{\Omega}\partial_tu \Delta \rho^3\, dx+\int_{\Omega}(\partial_tu)^2\, dx
	+\int_{\Omega}\left(\Delta \rho^3\right)^2 \, dx=0.
	\end{equation}
	We calculate, with the aid of (\ref{ptu}), that 
	\begin{eqnarray}
	-2\int_{\Omega}\partial_tu \Delta \rho^3\, dx&=&-2\int_{\Omega}\partial_tu \Delta \left(\dut-\bt\right)\, dx-2\int_{\Omega}\partial_tu \Delta \bt\, dx\nonumber\\
	&=&-2\int_{\Omega}\partial_t\Delta u  \left(\dut-\bt\right)\, dx-2\int_{\Omega}\partial_tu \Delta \bt\, dx\nonumber\\
	&=&\frac{d}{dt}
	\int_{\Omega}\left[\duw+2\Delta u  \bt-2u \Delta \bt\right]\,dx.\label{e12}
	\end{eqnarray}
	Note that
	\begin{equation}
	|u(x,t)|\leq\int_{0}^{T}|\partial_t u(x,t)|dt+|u_0(x)|.
	\end{equation}
	Subsequently, we have
	\begin{equation}
	\io |u(x,t)|^2dx\leq c\iot|\partial_t u(x,t)|^2dxdt+c.\label{e13}
	\end{equation}
	Here and in what follows the letter $c$ denotes a positive number depending only on the given data. Substitute (\ref{e12}) into (\ref{e11}), integrate the resulting equation with respect $t$ over $(0, s)$, where
	$s\in (0, T]$, and keep in mind (\ref{e13}) and the assumption that $b_1(x)$ is bounded to derive 
	\begin{equation}
	\int_{\Omega} \left(\rho^2+\rho^{-1}\right)\, dx
	+\int_{\Omega_s}\left[(\partial_tu)^2+\left(\Delta\rho^3\right)^2\right] \, dx \, dt\leq c,\label{e15}
	\end{equation}
	where $\Omega_s=\Omega\times(0,s)$. This is our first a priori estimate.
	
	We can infer from the Calderon-Zygmund inequality that
	\begin{equation}\label{e14}
	\int_{0}^{T}\|\rho^3\|^2_{W^{2,2}(\Omega)}dt\leq c\int_{\Omega_T}\left(\Delta\rho^3\right)^2 \, dx \, dt
	+c\|\bt\|^2_{W^{2,2}(\Omega)}\leq c.
	\end{equation}
	
	To derive our second a priori estimate, we multiply each term in (\ref{p1}) by $\Delta u-b_1(x) $ and then
	integrate over $\Omega$ to obtain
	\begin{equation}\label{e21}
	\int_{\Omega}\partial_tu(\Delta u-b_1(x))\, dx
	=\int_{\Omega}\Delta\dut(\Delta u-b_1(x))\, dx.
	\end{equation}
	Keeping (\ref{ptu}) in mind, we compute
	\begin{eqnarray}
	\int_{\Omega}\partial_tu(\Delta u-b_1(x))\, dx&=&-\int_{\Omega}\partial_t\nabla u\nabla u\, dx-\int_{\Omega}\partial_tu\bo\, dx\nonumber\\
	&=&-\frac{d}{dt}\int_{\Omega}\left(\frac{1}{2}|\nabla u|^2+u\bo\right)\, dx.\label{e22}
	\end{eqnarray}
	The right-hand side of (\ref{e21}) can be estimated as follows.
	\begin{eqnarray}
	\lefteqn{\int_{\Omega}\Delta\dut(\Delta u-b_1(x))\, dx}\nonumber\\
	&=&-\int_{\Omega}\nabla\dut\nabla(\Delta u-b_1(x))\, dx\nonumber\\
	&=&3\int_{\Omega}\left(\Delta u\right)^{-4}|\nabla\Delta u|^2\, dx+\int_{\Omega}\nabla\left(\dut-\bt\right)\nabla\bo\, dx\nonumber\\
	&&+\int_{\Omega}\nabla\bt\nabla\bo\, dx\nonumber\\
	&=&3\int_{\Omega}|\nabla\rho|^2\, dx-\int_{\Omega}\left(\dut-\bt\right)\Delta\bo\, dx\nonumber\\
	&&+\int_{\Omega}\nabla\bt\nabla\bo\, dx\label{e23}
	\end{eqnarray}
	Use (\ref{e22}) and (\ref{e23}) in (\ref{e21}), integrate the resulting equation with respect to $t$, 
	and thereby obtain
	\begin{equation}
	\int_{\Omega}|\nabla u|^2\, dx+\int_{\Omega_s}|\nabla\rho|^2\, dx \, dt\leq c\int_{\Omega_s}\rho^6\, dx \, dt+c\int_{\Omega}u^2 \, dx+c\leq c.\label{e24}
	\end{equation}
	The last step is due to (\ref{e14}) and (\ref{e13}).
	
	For more a priori estimates,
 we differentiate (\ref{p1}) with respect to $t$, multiply through the resulting equation by $\partial_t u$, and
	integrate over $\Omega$ to obtain
	\begin{eqnarray}
	\frac{1}{2}\frac{d}{dt}\int_{\Omega}(\partial_t u)^2\, dx
	&=&\int_{\Omega}\partial_t\Delta\dut\, \partial_t u\, dx\nonumber\\
	&=&\int_{\Omega}\partial_t\dut\,\partial_t\Delta u\, dx\\
	&=&-3\int_{\Omega}\left(\Delta u\right)^{-4} \, |\partial_t\Delta u|^2 \, dx\label{e31}
	=-3\int_{\Omega_s}|\partial_t \rho|^2\, dx
	\end{eqnarray}
	Here we have used (\ref{ptu}) and the fact that
	\begin{equation}
	\partial_t\dut=0 \ \ \mbox{on $\Sigma_T$.}
	\end{equation}
	Integrating (\ref{e31}) with respect to $t$ gives
	\begin{eqnarray}
	\frac{1}{2}\int_{\Omega}|\partial_t u(x,s)|^2\,dx
	+3\int_{\Omega_s}|\partial_t \rho|^2 \, dx\, dt&=&	
	\frac{1}{2}\int_{\Omega}|\partial_t u(x,0)|^2 \, dx\nonumber\\
	&=&\frac{1}{2}\int_{\Omega}|\Delta \bt|^2\, dx.\label{e32}
	\end{eqnarray}
	This together with (\ref{p1}) and (\ref{cr1}) yields
	\begin{equation}
	\rho^3\in L^\infty(0,T; W^{2,2}(\Omega)),\ \ \ \rho^2\Delta^2\rho^3\in L^2(\Omega_T).\label{ssm22}
	\end{equation}	 	
	

These a priori estimates 
		motivate the 
	following definition of a weak solution.

\noindent {\bf Definition 1.1.} 
We say that a pair $(u,\rho)$ is a weak solution to (\ref{p1})-(\ref{p4}) if the following conditions hold:
\begin{enumerate}
	\item[(D1)]
	$\rho\in W^{1,2}(\omt), \rho^3\in L^\infty(0,T;W^{2,2}(\Omega)), \rho\geq 0$ a.e. on $\omt$, and $\frac{1}{\rho}\in  L^\infty(0,T;L^{1}(\Omega))$;
	\item[(D2)] 
	$u\in C([0,T]; L^2(\Omega)),  
	\frac{\partial u}{\partial t}, |\nabla u|\in  L^\infty(0,T;L^{2}(\Omega))$;
	 \item[(D3)] We have
	 \begin{equation}
	 \Delta u=\frac{1}{\rho}+\nu_s,
	 \end{equation} where $\nu_s\in L^2(0,T; W^{-1,2}(\Omega))$ is a non-negative, finite Radon measure on $\omt$ whose support is contained in the set $A_0$, where
	 \begin{equation}
	 A_0=\{(x,t)\in \omt: \rho(x,t)=0\};
	 \end{equation}
	\item[(D4)] There hold
	\begin{eqnarray}
	\frac{\partial u}{\partial t} &=& \Delta \rho^3 \ \ \ \mbox{a.e. on $\omt$,}\label{ow11}\\
	 u &=& \bz \ \ \ \mbox{a.e. on $\Sigma_T$},\\
 \rho &=& \frac{1}{\bo}\ \ \ \mbox{a.e. on $\Sigma_T$}.
	\end{eqnarray}
	The initial condition (\ref{p4}) is satisfied in the space $C([0,T]; L^2(\Omega))$.
\end{enumerate}

Obviously, (D3) is the decomposition of $\Delta u$ with respect to the Lebesgue measure (\cite{EG}, p.42) and $\frac{1}{\rho}$ turns out to be the absolutely continuous part. According to the analysis in \cite{LX}, the possible existence of a singular part in $\Delta u$ is an intrinsic property of our weak solutions. No matter how smooth $\rho$ is, as long as $A_0$ is not empty we can have  $\nu_s\ne 0$.

\subsection{Main results}
Our main result is the following 
\begin{thm}	Let $\Omega$ be a bounded domain in $\mathbb{R}^N$ with $C^2$ boundary and assume that (H1)-(H3) hold. Then there is a 
	weak solution to (\ref{p1})-(\ref{p4}) in the sense of Definition 1.1. Furthermore, the uniqueness assertion holds for these solutions whose Laplacians have no singular parts. That is, there is only
	one weak solution $u$ to (\ref{p1})-(\ref{p4}) with $\nu_s=0$.
\end{thm}

For each $T>0$ there is a weak solution $u$ on $\Omega_T$. In this sense, our weak solution is global.
For the uniqueness part of Theorem 1.1, we refer the reader to \cite{LX}.

Does $\rho$ obtained in Theorem 1.1 satisfy (\ref{cr1}) in a suitable sense? That is, we want to find out whether (\ref{ow11}) combined with (D3) implies (\ref{cr1}). If $\rho$ is continuous on $\omt$, the answer is yes.
To see this, we first note that the set $\omt\setminus A_0$ is open. Then we can infer from (D3) that
 \begin{equation}
 \Delta u =\frac{1}{\rho}\ \ \mbox{in $\mathcal{D}^\prime(\omt\setminus A_0)$, and therefore everywhere on $\omt\setminus A_0$.}
 \end{equation}
 Observe that we have
 \begin{equation}
 \frac{1}{\rho}\in L^\infty_{\mbox{loc}}(\omt\setminus A_0)\cap W^{1,2}_{\mbox{loc}}(\omt\setminus A_0).
 \end{equation}
Take the Laplacian of both sides of (\ref{ow11}) to yield
 \begin{equation}
 -\frac{1}{\rho^2}\partial_t\rho=\Delta^2\rho^3\ \ \mbox{in $\omt\setminus A_0$.}
 \end{equation}
Thus $\Delta^2\rho^3$ exists as a function on $\omt\setminus A_0$ and $\rho^2\Delta^2\rho^3=-\partial_t\rho\in L^2(\omt\setminus A_0)$. Since $A_0$ has Lebesgue measure $0$, we have
\begin{equation}
\partial_t\rho+\rho^2\Delta^2\rho^3=0\ \ \mbox{a.e. on $\omt$.}
\end{equation}


 If $\Omega=(0,1)$, then (D1) combined with Lemma 2.2 in Section 2 asserts that $\rho$ is continuous on $\omt$.
Thus our theorem implies the existence result in \cite{GLL2}. At the end of Section 3, we will briefly indicate how we can use our approximate scheme to construct a sequence of approximate solutions that are uniformly H\"{o}lder continuous on $\overline{\omt}$ in the one-dimensional case.

If $\nu_s=0$, the answer is also yes. In this case we can deduce from  (\ref{ow11}) that
\begin{equation}
\partial_t\left(\frac{1}{\rho}\right)=\Delta^2\rho^3\ \ \mbox{in $\mathcal{D}^\prime(\omt)$.}\label{osd1}
\end{equation}
Let $\xi$ be a $C^\infty$ function on $\mathbb{R}^N\times\mathbb{R}$ with the property
\begin{equation}\label{bcd}
\xi\mid_{\partial\Omega_T}=0.
\end{equation}
It is not difficult to see that we can use $\xi\rho^3$ as a test function in the above equation, from whence follows
\begin{equation}
-\int_{\Omega_T}\rho\partial_t\rho\xi dxdt
=\int_{\Omega_T}\Delta\rho^3\Delta\left(\rho^3\xi\right) dxdt.\label{osa2}
\end{equation}
With the aid of (\ref{bcd}), we can use the preceding equation as a weak formulation for (\ref{cr1})  and (\ref{cr3}).

If $\rho$  satisfies the additional integrability conditions
\begin{equation}
\rho\partial_t\rho\in L^2(0,T;W_0^{1,2}(\Omega)), \ \ \rho^2\in L^2(0,T;W^{2,2}(\Omega)),\label{ssd1}
\end{equation} 
then the answer is still yes. In this case, (\ref{ow11}) implies
\begin{equation}
\partial_t\left(\frac{1}{\rho}+\nu_s\right)=\Delta^2\rho^3\ \ \mbox{in $\mathcal{D}^\prime(\omt)$.}\label{ssm11}
\end{equation}
Let $\xi$ be given as before.
The function  $\rho^2\xi$ is a legitimate test function for (\ref{ssm11}). Upon using it, we arrive at
\begin{equation}
-\int_{\Omega_T}\partial_t\rho\xi dxdt
-\int_{\Omega_T}\left(2\rho\partial_t\rho\xi+\rho^2\partial_t\xi\right)d\nu_s
=\int_{\Omega_T}\Delta\rho^3\Delta\left(\rho^2\xi\right) dxdt.\label{vi1}
\end{equation}
The item (D3) asserts that the second integral in the above equation is $0$. 
This leads to the equation
\begin{equation}
-\int_{\Omega_T}\partial_t\rho\xi dxdt
=\int_{\Omega_T}\left(\Delta\rho^3\Delta\rho^2\xi+2\Delta\rho^3\nabla\rho^2\nabla\xi+\Delta\rho^3\rho^2\Delta\xi\right) dxdt.
\end{equation}
As before, this equation can be viewed as a weak formulation for (\ref{cr1}) and (\ref{cr3}).
Unfortunately, we have not been able to deduce (\ref{ssd1}).
Thus we seek to further weaken the notion of a solution for (\ref{cr1})-(\ref{cr4}).

In view of the a priori estimates we have been able to derive, we give the following definition for (\ref{cr1})-(\ref{cr4}).

\noindent {\bf Definition 1.2.} 
We say that a function $\rho$ is a weak solution to (\ref{cr1})-(\ref{cr4}) if, in addition to (D1), (\ref{cr2}), and (\ref{cr4}) , we have  
\begin{equation}\label{ssa23}
\int_{\Omega_T}\left(\partial_t\rho\rho+\left(\Delta\rho^3\right)^2\right)\xi dxdt+2\int_{\Omega_T}\Delta\rho^3\nabla\rho^3\nabla\xi dxdt+\int_{\Omega_T}\rho^3\Delta\rho^3\Delta\xi dxdt=0
\end{equation}
for all $\xi\in C^\infty\left(\mathbb{R}^N\times\mathbb{R}\right)$ with the property
$$\xi\mid_{\Sigma_T}=0.$$

It is easy to see that if $\rho $ is sufficiently smooth, the above equation implies (\ref{cr1}) and the boundary condition (\ref{cr3}). Since
$\rho\in W^{1,2}(\omt)$, (\ref{cr4}) makes sense. 
We would like to remark that the model derived in \cite{X} has been analyzed in \cite{DFL} via a variational inequality formulation. It is possible to cast (\ref{p1}) as the gradient flow of a suitably-defined functional in a Hilbert space. We will pursue this possibility in a future study. However, the preceding approaches do not seem to work for (\ref{cr1}). 

In general, we have the following result.
\begin{thm}The function $\rho$ obtained in Theorem 1.1 satisfies the inequality 
	\begin{equation}\label{ssa22}
	\int_{\Omega_T}\left(\partial_t\rho\rho+\left(\Delta\rho^3\right)^2\right)\xi dxdt+2\int_{\Omega_T}\Delta\rho^3\nabla\rho^3\nabla\xi dxdt+\int_{\Omega_T}\rho^3\Delta\rho^3\Delta\xi dxdt\leq 0
	\end{equation}
	for all $\xi\in C^\infty\left(\mathbb{R}^N\times\mathbb{R}\right)$ with the properties
	$$\xi\geq 0,\ \ \ \xi\mid_{\Sigma_T}=0.$$
\end{thm}
It would be interesting to know  if the function $\rho$ in Theorem 1.1 satisfies the equality in (\ref{ssa22}) in general. If it did,  it would be a weak solution to (\ref{cr1})-(\ref{cr4}) in the sense of Definition 1.2. Unfortunately, we have not been able to prove it under the assumptions of Theorem 1.1.


Fourth-order nonlinear parabolic equations arise in a variety of physical settings.
Two well-known examples are the thin film equation \cite{BF} and the quantum drift-diffusion model \cite{J2}.
A well-known difficulty in the study of  fourth-order equations is that the maximum principle is no longer valid for such equations. As a result, the techniques one often uses in the analysis of second-order equations are mostly not applicable.
From the point of view of mathematical analysis, it is a little bit surprising that (\ref{cr1})-(\ref{cr4}) can have a non-negative solution. There is really not much in the structure of the equation (\ref{cr1}) that indicates the existence of such a solution. One can easily see that if $\rho$ is a solution, so is $-\rho$.
Keep in mind that the fundamental solution for the biharmonic heat equation changes signs, and sign change
is an intrinsic property of solutions to linear fourth-order parabolic equations. That is to say, we have to rely on the nonlinear structure of our equation for possible existence of a non-negative solution. In this regard, the nonlinearity in (\ref{cr1}) does not seem to offer much help.
 This probably explains why we have not been able to construct a sequence of positive approximate solutions to (\ref{cr1})-(\ref{cr4}) directly in high space dimensions. (It can be done in the one-dimensional case. See \cite{GLL2} for details.) What saves us is the discovery of the connection between (\ref{p1}) and (\ref{cr1}), and we can study the latter via the former.


		A solution to (\ref{p1})-(\ref{p4}) will be constructed as the limit of a sequence of approximate solutions. The key is to design an approximation scheme so that all the calculations in the derivation of (\ref{e15}), (\ref{e24}), and (\ref{e32}) can be justified. This is accomplished in Sections 2 and 3. To be more specific,
		in Section 2 we state a couple of preparatory lemmas and present our approximate 
		problem. The trick here is the introduction of a suitable substitution. 
		This substitution takes care of the sign issue for $\Delta u$ in (\ref{p1})-(\ref{p4}), and  we also obtain the existence of a classical solution to the approximate problem. We form a sequence of approximate solutions based upon implicit discretization in the time variable. Section 3 is devoted to the proof that the estimates (\ref{e15}), (\ref{e24}), (\ref{e32}), and (\ref{ssm22}) are all preserved for the sequence except the second one in (\ref{ssm22}), and this is enough to justify passing to the limit. 
	\section{Approximate Problems}
	\setcounter{equation}{0}
	Before we present our approximate problem,  we collect a few frequently used elementary inequalities in the following lemma.
	\begin{lem} we have:
		\begin{enumerate}
			\item[(L1)] $x\cdot(x-y)\geq \frac{1}{2}(|x|^2-|y|^2)$ for $x,y\in \mathbb{R}^N$;
			\item[(L2)] if $f$ is an increasing (decreasing) function on $\mathbb{R}$ and $F$ an anti-derivative of $f$, then
			$$f(s)(s-t)\geq (\leq)F(s)-F(t)
			\ \ \mbox{ for all $s,t\in \mathbb{R}$};
			$$
		\item[(L3)]	$(e^{-s}-e^{-t})e^{3s}\leq -\frac{1}{2}(e^{2s}-e^{2t})$ for all $s,t\in \mathbb{R}$;
			\item[(L4)] 
			$(a^3-b^3)(a^{-1}-b^{-1})\leq -3(a-b)^2$, where $a, b\in (0,\infty)$.
		\end{enumerate}
	\end{lem}
	Obviously, only (L3) and (L4) deserve some attention. 
	To establish (L3), we set
	$$s_1=e^{-s},\ \  s_2=e^{-t}.
	$$
Then we compute from (L2) that
	\begin{eqnarray*}
	(e^{-s}-e^{-t})e^{3s}&=&(s_1-s_2)s_1^{-3}\\
	&\leq&\int_{s_2}^{s_1}s^{-3}ds=-\frac{1}{2}(s_1^{-2}- s_2^{-2})\\
	&=&-\frac{1}{2}(e^{2s}-e^{2t}).
	\end{eqnarray*}
	The proof of (L4) is also rather elementary. For $a, b\in (0,\infty)$ we compute
	\begin{eqnarray*}
	(a^3-b^3)(a^{-1}-b^{-1})&=& (a-b)(a^2+ab+b^2)\frac{b-a}{ab}\\
	&=&-(a-b)^2\left(\frac{a}{b}+\frac{b}{a}+1\right)\\
	&\leq&-3(a-b)^2.
	\end{eqnarray*}
	The one-dimensional existence theorem is based upon the following lemma, which is inspired by (\cite{E}, p. 288) and \cite{GLL2}.
	\begin{lem}
		If $\omt=(0,1)\times(0,T)$ and $f\in W^{1,2}(\omt)\cap L^\infty(0,T; W^{2,2}(\Omega))$ with the boundary conditions
		\begin{equation}
		f(0,t)=a,\ \  f(1,t)=b, \ a, b\in \mathbb{R},
		\end{equation}
		then there is a positive number $c=c(\|\partial_t f\|_2, \|f\|_{ L^\infty(0,T; W^{2,2}(\Omega))})$ such
		that
		\begin{equation}
		|f(x_1, t_1)-f(x_2, t_2)|\leq c\left(|x_1-x_2|^{\frac{1}{2}}+|t_1-t_2|^{\frac{1}{4}}\right)\ \ \mbox{for all $(x_1,t_1), (x_2,t_2) \in \omt$}.
		\end{equation}
		Here and in what follows $\|\cdot\|_p$ denotes the norm in the space $L^p(\omt)$.
	\end{lem}	
	\begin{proof} Let $(x_1,t_1), (x_2,t_2) \in \omt$ be given.
		We calculate from the first boundary condition that
		\begin{eqnarray}
		|f(x_1, t_1)-f(x_2, t_2)|&=&\left|\int_{0}^{x_1}\partial_sf(s,t_1)ds-\int_{0}^{x_2}\partial_sf(s,t_2)ds\right|\nonumber\\
		&\leq& \int_{0}^{1}\left|\partial_sf(s,t_1)-\partial_sf(s,t_2)\right|ds+\left|\int_{x_1}^{x_2}|\partial_sf(s,t_2)|ds\right|\nonumber\\
		&\leq&\left(\int_{0}^{1}\left|\partial_sf(s,t_1)-\partial_sf(s,t_2)\right|^2ds\right)^{\frac{1}{2}}+c|x_1-x_2|^{\frac{1}{2}}.\label{sm111}
		\end{eqnarray}
		Note that our boundary conditions imply 
		$f(s,t_1)-f(s,t_2)=0$ for $s=0$ and $1$. Keeping this in mind, we apply the integration-by-parts formula to obtain
		\begin{eqnarray}
		\lefteqn{\int_{0}^{1}\left|\partial_sf(s,t_1)-\partial_sf(s,t_2)\right|^2ds}\nonumber\\
		&=&-	\int_{0}^{1}(\partial^2_{ss}f(s,t_1)-\partial^2_{ss}f(s,t_2))(f(s,t_1)-f(s,t_2))ds\nonumber\\
		&\leq&c\left(\int_{0}^{1}\left|f(s,t_1)-f(s,t_2)\right|^2ds\right)^{\frac{1}{2}}\nonumber\\
		&\leq&c\left(\int_{0}^{1}\left|\int_{t_2}^{t_1}\partial_t f(s,t)dt\right|^2ds\right)^{\frac{1}{2}}\nonumber\\
		&\leq& c|t_1-t_2|^{\frac{1}{2}}.
		\end{eqnarray}
		Plugging this into (\ref{sm111}) yields the desired result.
		\end{proof}	
	To design our approximate problem, we introduce a new unknown function $\psi$ so that
\begin{equation}
\Delta u=\ens.
\end{equation}
Obviously, this will force $\Delta u$ to be non-negative. Of course, there are trade-offs from doing this,
 and we will see them later. Now (\ref{p1}) becomes
\begin{equation}
\partial_tu=\Delta\ets.
\end{equation}
We employ an implicit discretization scheme in the time variable for this equation. This leads
to the consideration of
the following system of two second-order elliptic equations
\begin{eqnarray}
-\Delta e^{3\psi}+\tau\psi &=&- \uvt\ \ \ \mbox{in $\Omega$},\label{ap1}\\
-\Delta u &=&-\ens\ \ \ \mbox{in $\Omega$}\label{ap2}
\end{eqnarray}
coupled with the boundary conditions
\begin{eqnarray}
 u&=& b_0(x)\ \ \ \mbox{on $\partial\Omega$},\label{ap3}\\
\psi&=&-\ln\bo\ \ \ \mbox{on $\partial\Omega$},\label{ap4}
\end{eqnarray}
where $\tau>0$ is the step size and $v$ is either $u_0$ or the solution obtained in a preceding step in the scheme.
The term $\tau\psi$ in (\ref{ap1}) has a regularizing effect, and, as we shall see, it helps to overcome the 
possible degeneracy caused by the exponential nonlinearity in the equation. Thus it is essential to  
 the existence of a solution to (\ref{ap1})-(\ref{ap4}). 
 In addition to the non-homogeneous Dirichlet boundary conditions considered here, another key difference between our approximate problem here and the one in \cite{LX} is due to the exponential term on the right-hand side of (\ref{ap2}). It turns out that this term renders many calculations in \cite{LX} invalid here. Thus we must find a new way of deriving a priori estimates.
The construction of our approximate solutions will be based upon this problem.

	\begin{prop}
		Let $\Omega$ be a bounded domain in $\mathbb{R}^N$ with Lipschitz boundary. Assume:
		\begin{enumerate}
			\item[(H4)]  $b_0(x), \bo\in W^{1,2}(\Omega)\cap \hc$, where $\hc$ denotes the space of all H\"{o}lder continuous functions on $\overline{\Omega}$, i.e., $u\in\hc$ if and only if $u\in C^{0,\alpha}(\overline{\Omega})$ for some $\alpha\in (0,1)$, and there is a positive number $c_0$ with the property
			\begin{equation}
			\bo\geq c_0\ \ \mbox{ on $\Omega$;}
			\end{equation}
			\item[(H5)] $v\in L^\infty(\Omega)$.
		\end{enumerate} Then there is a 
		weak solution to (\ref{ap1})-(\ref{ap4}). If, in addition, $\partial\Omega$ is $C^2$, $v\in\hc, b_0(x) , \bo\in C^{2,\alpha}(\overline{\Omega})$ for some $\alpha\in (0,1)$, then the solution is also classical.
	\end{prop}
	\begin{proof}
	The existence assertion will be established via the Leray-Schauder Theorem
	(\cite{GT}, p. 280). For this purpose, we define an operator $B$ from $L^\infty(\Omega)$
	into itself as follows: for each $g\in L^\infty(\Omega)$ we say $B(g)=\psi$ if $\psi$ solves the  boundary value problem
	\begin{eqnarray}
	-\mbox{div}\left(3e^{3g}\nabla\psi\right)+\tau\psi &=&-\frac{u-v}{\tau} \ \ \ \mbox{in $\Omega$},\label{om3}\\
	\psi&=&-\ln\bo\ \ \ \mbox{on $\partial\Omega$},\label{om4}
		\end{eqnarray}
		where $u$ satisfies the linear problem
		\begin{eqnarray}
		-\Delta u &=&-e^{-g}\ \ \ \mbox{in $\Omega$},\label{om5}\\
		 u&=& b_0(x)\ \ \ \mbox{on $\partial\Omega$}.\label{om6}
		\end{eqnarray}
To see that $B$ is well-defined, we observe from the classical existence and regularity theory for linear elliptic equations (\cite{GT}, Chap. 8) that (\ref{om5})-(\ref{om6}) has a unique solution $u$ in the space $W^{1,2}(\Omega)\cap \hc$.
Since $g\in L^\infty(\Omega)$ the elliptic coefficient
in (\ref{om3}) is bounded above and bounded away from $0$ below. Thus we can conclude that (\ref{om3})-(\ref{om4}) also has a unique solution $\psi$ in the space $W^{1,2}(\Omega)\cap \hc$. 
In view of these results, we can claim that $B$ is well-defined, continuous, and maps bounded sets into precompact ones. It remains to show that there is a positive number $c$ such that
\begin{equation}
\|\psi\|_\infty\leq c\label{ot8}
\end{equation}
for all $\psi\in L^\infty(\Omega)$ and $\sigma\in (0,1]$ satisfying
$$\psi=\sigma B(\psi).$$
This equation is equivalent to the boundary value problem
\begin{eqnarray}
	-\mbox{div}\left(3e^{3\psi}\nabla\psi\right)+\tau\psi  &=&-\sigma\frac{u-v}{\tau} \ \ \ \mbox{in $\Omega$},\label{ot9}\\
-\Delta u &=&-\ens \ \ \ \mbox{in $\Omega$},\label{ot10}\\
 u&=& b_0(x)\ \ \ \mbox{on $\partial\Omega$},\\
 \psi&=& -\sigma\ln\bo\ \ \ \mbox{on $\partial\Omega$}.
\end{eqnarray}
Use $(u-\|b_0(x)\|_\infty)^+$ as a test function in (\ref{ot10}) to obtain
\begin{equation}
u\leq \|b_0(x)\|_\infty \ \ \mbox{on $\Omega$.}
\end{equation}
The term on the right-hand side of (\ref{ot9}) satisfies
\begin{eqnarray}
 -\sigma\frac{u-v}{\tau}&\geq&\frac{\sigma v}{\tau}-\frac{\sigma \|b_0(x)\|_\infty}{\tau}\nonumber\\
 &\geq& -\frac{1}{\tau}\left(\min_\Omega v-\|b_0(x)\|_\infty\right)^- \ \ \ \mbox{on $\Omega$}.\label{sm5}
\end{eqnarray}
On the boundary $\partial\Omega$, we have
\begin{equation}
 \psi=-\sigma\ln\bo\geq -\sigma\ln\|\bo\|_\infty\geq -\left(\ln\|\bo\|_\infty\right)^+.\label{sm2}
\end{equation}
We pick a number $M$ with the properties 
\begin{eqnarray}
\tau M &\geq &\frac{1}{\tau}\left(\| v\|_\infty+\|b_0(x)\|_\infty\right) ,\label{sm6}\\
M&\geq &\left(\ln\|\bo\|_\infty\right)^+.\label{sm1}
\end{eqnarray}
Adding $\tau M$ to both sides of (\ref{ot9}) yields
\begin{equation}
-\mbox{div}\left(3e^{3\psi}\nabla\psi\right)+\tau(\psi+M)  =-\sigma\frac{u-v}{\tau}+\tau M\geq 0 \ \ \ \mbox{in $\Omega$.}\label{sm3}
\end{equation}
The last step is due to (\ref{sm5}) and (\ref{sm6}). It follows from (\ref{sm1}) and (\ref{sm2}) that $(\psi+M)^-=0 $ on $\partial\Omega$. Use $(\psi+M)^-$ as a test function in (\ref{sm3}) to derive
\begin{equation}
-\int_{\Omega}3\ets|\nabla(\psi+M)^-|^2dx-\tau\int_{\Omega}|(\psi+M)^-|^2dx\geq 0.
\end{equation}
This implies
\begin{equation}
\psi\geq -M \ \ \mbox{on $\Omega$.}\label{st3}
\end{equation}
Consequently, the term on the right-hand of (\ref{ot10}) satisfies
\begin{equation}
-e^M\leq- e^{-\psi}\leq 0 \ \ \mbox{on $\Omega$.}
\end{equation}
This puts us in a position to apply the classical $L^\infty$ estimate (\cite{GT}, p. 189), from whence follows
\begin{equation}
\|u\|_\infty\leq c\|b_0(x)\|_\infty+c\|\nabla b_0(x)\|_2+c\|e^{-\psi}\|_\infty\leq c.
\end{equation}
Combining this with our assumption on $v$ yields
\begin{equation}
\|u-v\|_\infty\leq c.
\end{equation}
	However, we can no longer use the classical $L^\infty$ estimate for (\ref{ot8}) because the elliptic coefficient in (\ref{ot9}) depends on the up bound of $\psi$. Instead, we are saved by 
	an argument similar to the proof of (\ref{st3}). To see this,
we choose a number $L$ so large that it satisfies
\begin{eqnarray}
\frac{1}{\tau}\|u-v\|_\infty &\leq&\tau L,\label{st1}\\
(\ln c_0)^-&\leq &L.\label{st2}
\end{eqnarray}
Note that (\ref{st1}) implies
\begin{equation}
-\sigma\frac{u-v}{\tau}-\tau L\leq 0\ \ \mbox{in $\Omega$,}
\end{equation}
while (\ref{st2}) asserts $(\psi-L)^+=0$ on $\partial\Omega$. 
Subtract $\tau L$ from both sides of (\ref{ot9}) and use $(\psi-L)^+$ as a test function in the resulting equation to derive
\begin{equation}
\int_{\Omega}3\ets|\nabla(\psi-L)^+|^2dx+\tau\int_{\Omega}|(\psi-L)^+|^2dx\leq 0.
\end{equation}
This together with (\ref{st3}) gives (\ref{ot8}). 

Obviously, the preceding calculations rely on the
term $\tau\psi$ in (\ref{ot9}). Thus it is essential to the existence proof.

The last part of the theorem ia s consequence of a rather standard bootstrap argument and the Schauder estimates (\cite{GT}, Chap. 6).
 One bootstraps from $\psi\in L^\infty(\Omega)$. 
We shall omit the details. The proof is complete.
\end{proof}
	\section{Proof of Theorems 1.1 and 1.2}
	\setcounter{equation}{0}
	
	The proof of Theorem 1.1 will be divided into several propositions. To begin with, we present our approximation scheme. This is based upon Proposition 2.1. Then
	we proceed to derive estimates similar to (\ref{e15}), (\ref{e24}), and (\ref{e32}) for our approximate problems. These estimates are shown to be sufficient to justify passing to the limit.

	Let $T>0$ be given. For each $j\in\{1,2,\cdots,\}$ we divide the time interval $[0,T]$ into $j$ equal subintervals. Set
	$$\tau=\frac{T}{j}.$$
	We discretize  (\ref{p1})-(\ref{p4}) as follows. For $k=1,\cdots, j$, we solve recursively the system
	\begin{eqnarray}
	\frac{u_k-u_{k-1}}{\tau}-\Delta\esk+\tau\psi_k&=&0\ \ \ \mbox{in $\Omega$},\label{s31}\\
	\Delta\uk&=&e^{-\psi_k}\ \ \ \mbox{in $\Omega$},\label{s32}\\
	\uk&=&\bzj\ \ \ \mbox{on $\partial\Omega$},\label{s33}\\
	\psi_k&=&-\ln\boj\ \ \ \mbox{on $\partial\Omega$},\label{s34}
	\end{eqnarray}
	where the sequences $\{u_{0j}(x)\}, \{\bzj\}$, and $\{\boj\}$ are the respective mollifications of $u_0(x), \bz, \bo$. 
	Thus $\{\boj\}$ is bounded and convergent in $W^{2,2}(\Omega)\cap L^\infty(\Omega)$ and satisfies
	\begin{equation}
	\boj\geq c_0\ \ \mbox{on $\Omega$ for each $j$}.\label{blb}
	\end{equation}
	Similar statements can be made of the sequences  $\{u_{0j}(x)\}$ and $\{\bzj\}$.
	
	Introduce the functions
	\begin{eqnarray}
	\utj(x,t)&=&\frac{t-t_{k-1}}{\tau}\uk(x)+\frac{t_k-t}{\tau}\uko(x), \ x\in\Omega,  \ t\in(t_{k-1},t_k],\label{dut}\\
	\ubj(x,t)&=&\uk(x), \  \ x\in\Omega, \ t\in(t_{k-1},t_k],\label{dub}\\
	\sbj(x,t)&=&\skx, \ \ \ x\in\Omega, \ \ t\in(t_{k-1},t_k],\\
	\rbj(x,t)&=&\rk(x)\equiv e^{\psi_{k}(x)}, \ \ x\in\Omega, \ t\in(t_{k-1},t_k],\\
	\rtj(x,t)&=&\frac{t-t_{k-1}}{\tau} e^{\psi_{k}(x)}+\frac{t_k-t}{\tau}e^{\psi_{k-1}(x)}, \ x\in\Omega,  \ t\in(t_{k-1},t_k],
	\end{eqnarray}
	where $t_k=k\tau$ and
	\begin{equation}
	\psi_0(x)=-\ln\Delta u_{0j}(x).
	\end{equation} We can rewrite (\ref{s31})-(\ref{s33}) as
	\begin{eqnarray}
	\frac{\partial\utj}{\partial t}-\Delta \esbj+\tau\sbj&=&0\ \ \ \mbox{in $\omt$},\label{omm1}\\
	\Delta\ubj &=&e^{-\sbj}\ \ \ \ \mbox{in $\omt$}.\label{omm2}
	\end{eqnarray}
	We proceed to derive a priori estimates for the sequence of approximate solutions $\{\utj,\ubj,\sbj\}$. 
	\begin{prop}
		We have
			\begin{equation}
			\max_{0\leq t\leq T}\int_\Omega \left(e^{2\overline{\psi}_j(x,t)}+e^{-\overline{\psi}_j(x,t)}\right)\, dx
			+\int_{\Omega_T}\left(\Delta\esbj\right)^2\, dxdt+\tau\int_{\Omega_T}e^{3\overline{\psi}_j}|\nabla\overline{\psi}_j|^2\, dxdt\leq c.
				\end{equation}
				Here and in what follows $c$ is a positive number that depends only on the given data.
	\end{prop}
	\begin{proof}
		Take the Laplacian of both sides of (\ref{s31}) and substitute (\ref{s32}) into the resulting
		equation to obtain
		\begin{equation}\label{sfa1}
		\frac{e^{-\psi_k}-e^{-\psi_{k-1}}}{\tau}-\Delta^2\esk+\tau\Delta\psi_k=0\ \ \ \mbox{in $\Omega$.}
		\end{equation}
	Multiply through the above equation by $\left(\esk-\btj\right)$ and integrate the resulting equation over
	$\Omega$ to yield
			\begin{eqnarray}
		\lefteqn{	\int_\Omega\frac{e^{-\psi_k}-e^{-\psi_{k-1}}}{\tau}\left(\esk-\btj\right)\, dx}\nonumber\\
		&&-\int_\Omega\Delta^2\esk\left(\esk-\btj\right)\, dx+\tau\int_\Omega\Delta\psi_k\left(\esk-\btj\right)\, dx=0.\label{sw2}
			\end{eqnarray}
	We estimate each integral on the left-hand side of the above equation. We begin with the first integral.
	This will be done via an application of (L3) as follows:
	\begin{eqnarray}
	\lefteqn{\int_\Omega\frac{e^{-\psi_k}-e^{-\psi_{k-1}}}{\tau}\left(\esk-\btj\right)\, dx}\nonumber\\
	&=&\int_\Omega\frac{e^{-\psi_k}-e^{-\psi_{k-1}}}{\tau}\esk\, dx-\int_\Omega\frac{e^{-\psi_k}-e^{-\psi_{k-1}}}{\tau}\btj\, dx\nonumber\\
	&\leq& -\frac{1}{2}\int_\Omega\frac{e^{2\psi_k}-e^{2\psi_{k-1}}}{\tau}\, dx-\int_\Omega\frac{e^{-\psi_k}-e^{-\psi_{k-1}}}{\tau}\btj\, dx.\label{it1}
	\end{eqnarray}	
	To estimate the second integral, we first observe from (\ref{s31}) and (\ref{s34}) that
	\begin{equation}
	\Delta\esk\mid_{\partial\Omega}=-\tau\ln\boj.\label{sfa2}
	\end{equation}
	Keeping this in mind, we compute
	\begin{eqnarray}
	\lefteqn{-\int_\Omega\Delta^2\esk\left(\esk-\btj\right)\, dx}\nonumber\\
	&=&\int_\Omega\nabla\left(\Delta\esk+\tau\ln\boj\right)\nabla\left(\esk-\btj\right)\, dx\nonumber\\
	&&-\tau\int_\Omega\nabla\ln\boj\nabla\left(\esk-\btj\right)\, dx\nonumber\\
		&=&-\int_\Omega\left(\Delta\esk+\tau\ln\boj\right)\Delta\left(\esk-\btj\right)\, dx\nonumber\\
		&&+\tau\int_\Omega\Delta\ln\boj\left(\esk-\btj\right)\, dx\nonumber\\
			&\leq&-c\int_\Omega\left(\Delta\esk\right)^2\, dx
		+\varepsilon\int_\Omega\left(\esk\right)^2\, dx+c(\varepsilon),\ \ \varepsilon>0.\label{sw1}
	\end{eqnarray}
		The Carderon-Zygmund inequality implies
		\begin{equation}
		\|\esk\|^2_{W^{2,2}(\Omega)}\leq c\int_\Omega\left(\Delta\esk\right)^2\, dx+c\|\boj\|^2_{W^{2,2}(\Omega)}\leq c\int_\Omega\left(\Delta\esk\right)^2\, dx+c.\label{cz}
		\end{equation}
	Using this in (\ref{sw1}) and then
	 choosing $\varepsilon$ sufficiently small, we arrive at 
	 \begin{equation}
	 -\int_\Omega\Delta^2\esk\left(\esk-\btj\right)\, dx\leq -c\int_\Omega\left(\Delta\esk\right)^2\, dx+c.\label{it2}
	 \end{equation}
	Now we have reached the last integral in (\ref{sw2}). For this, we have
	\begin{eqnarray}
	\lefteqn{\tau\int_\Omega\Delta\psi_k\left(\esk-\btj\right)\, dx}\nonumber\\
&=&-\tau\int_\Omega\nabla\psi_k\nabla\left(\esk-\btj\right)\, dx\nonumber\\
	&=&-\tau\int_\Omega\nabla\psi_k\nabla\esk\, dx
	+\tau\int_\Omega\nabla\left(\psi_k+\ln\boj\right)\nabla\btj\, dx\nonumber\\
	&&-\tau\int_\Omega\nabla\ln\boj\nabla\btj\, dx\nonumber\\
	&\leq&-3\tau\int_\Omega\esk|\nabla\psi_k|^2\, dx-\tau\int_\Omega\left(\psi_k+\ln\boj\right)\Delta\btj\, dx+c\nonumber\\
	&\leq&-3\tau\int_\Omega\esk|\nabla\psi_k|^2\, dx+\int_\Omega\left(\psi_k\right)^2\, dx+c.\label{it3}
	\end{eqnarray}
	Using (\ref{it1}), (\ref{it2}), and (\ref{it3}) in (\ref{sw2}), multiplying through the resulting inequality by $-1$, we deduce
	\begin{eqnarray}
	\lefteqn{\frac{1}{2}\int_\Omega\frac{e^{2\psi_k}-e^{2\psi_{k-1}}}{\tau}\, dx+\int_\Omega\frac{e^{-\psi_k}-e^{-\psi_{k-1}}}{\tau}\btj\, dx}\nonumber\\
	&&+\int_\Omega\left(\Delta\esk\right)^2\, dx+\tau\int_\Omega\esk|\nabla\psi_k|^2\, dx\nonumber\\
	&\leq&c\int_\Omega\left(\psi_k\right)^2\, dx+c.
	\end{eqnarray}
	Multiply through this inequality by $\tau$ and sum up the resulting inequality over $k$ to derive
	\begin{eqnarray}
	\lefteqn{\int_\Omega\left(e^{2\overline{\psi}_j}+e^{-\overline{\psi}_j}\right)\, dx+\int_{\Omega_s}\left(\Delta e^{3\overline{\psi}_j}\right)^2\, dx\, dt}\nonumber\\
	&&+\tau\int_{\Omega_s}e^{3\overline{\psi}_j}|\nabla\overline{\psi}_j|^2\, dx\, dt\nonumber\\
	&\leq&c\int_{\Omega_s}\left(\overline{\psi}_j\right)^2\, dx\, dt+c.\label{sw3}
	\end{eqnarray}
	Here we have used the assumption (\ref{blb}). It is not difficult to see that the above inequality holds
	for each $s\in (0, T]$. We easily conclude that there is a positive number $K$ such that
	$$e^{2s}+e^{-s}+K\geq s^2\ \ \mbox{for each $s\in\mathbb{R}$}.$$
	That is,
	\begin{equation}
\int_\Omega\left(\overline{\psi}_j\right)^2dx\leq	\int_\Omega\left(e^{2\overline{\psi}_j}+e^{-\overline{\psi}_j}\right)\, dx+c.\label{slt}
	\end{equation}
	Using this in (\ref{sw3}), we deduce
	\begin{equation}
	\int_\Omega\left(e^{2\overline{\psi}_j}+e^{-\overline{\psi}_j}\right)\, dx\leq c\int_{\Omega_s}\left(e^{2\overline{\psi}_j}+e^{-\overline{\psi}_j}\right)\, dx\, dt+c.
	\end{equation}
	An application of Gronwall's inequality yields the proposition. The proof is complete.
		\end{proof}
		Note that we cannot mimic the calculations leading to (\ref{e15}) to prove this proposition.

		\begin{prop}
			There holds
			\begin{equation}
	\max_{0\leq t\leq T}\int_\Omega |\nabla\ubj|^2\, dx+
	\int_{\Omega_T}|\nabla e^{\sbj}|^2\, dxdt-\tau\int_{\{\sbj\leq 0\}}\sbj e^{-\sbj} \, dxdt
	\leq c.
					\end{equation}
		\end{prop}
		
		Obviously, this proposition is the discretized version of (\ref{e24}).
		
		\begin{proof} Multiply through  (\ref{s31}) by $\Delta u_k$ and then integrate over $\Omega$ to obtain
			\begin{equation}
			\int_\Omega\frac{ u_k- u_{k-1}}{\tau}\Delta\uk \, dx	-\int_\Omega\Delta\esk\Delta\uk \, dx
			+\tau\int_\Omega\psi_k \Delta\uk \, dx=0.\label{ota10}
			\end{equation}
			Note that 
			\begin{equation}
			( u_k- u_{k-1})\mid_{\partial\Omega}=0.
			\end{equation}The first integral in (\ref{ota10}) can be estimated as follows:
			\begin{eqnarray}
			\int_\Omega\frac{ u_k- u_{k-1}}{\tau}\Delta\uk \, dx&=&-\int_\Omega\frac{\nabla u_k-\nabla u_{k-1}}{\tau}\cdot\nabla\uk \, dx\nonumber\\
			&\leq& -\frac{1}{2\tau}	\int_\Omega\left(|\nabla u_k|^2-|\nabla u_{k-1}|^2\right)\, dx.\label{sf2}
			\end{eqnarray}
			Recalling (\ref{s32}), we compute the second integral in (\ref{ota10}) as follows:
			\begin{eqnarray}
				-\int_\Omega\Delta\esk\Delta\uk \, dx &=&	-\int_\Omega\Delta\esk e^{-\psi_k} \, dx\nonumber\\
				&=&-\int_\Omega\Delta\left(\esk -\btj\right)e^{-\psi_k}\, dx-\int_\Omega\Delta \btj e^{-\psi_k}\, dx\nonumber\\
				&=&\int_\Omega\nabla\left(\esk -\btj\right)\nabla e^{-\psi_k}\, dx\nonumber\\
				&&-\int_\Omega\Delta \btj \left(e^{-\psi_k}-\boj\right)\, dx\nonumber\\
				&&-\int_\Omega\Delta \btj \boj\, dx\nonumber\\
				&=&-3\int_\Omega e^{2\psi_k}|\nabla\psi_k |^2\, dx-\int_\Omega\nabla\btj\nabla e^{-\psi_k}\, dx\nonumber\\
				&&+\int_\Omega\nabla\btj \nabla\left(e^{-\psi_k}-\boj\right)\, dx\nonumber\\
				&&-\int_\Omega\Delta \btj \boj\, dx
				\leq -3\int_\Omega|\nabla  e^{\psi_k} |^2\, dx+c.\label{sf1}
					\end{eqnarray}
			Use (\ref{sf1}) and (\ref{sf2}) in (\ref{ota10}) and multiply through the resulting equation by $-1$ to derive
			\begin{equation}
			\frac{1}{2\tau}	\int_\Omega\left(|\nabla u_k|^2-|\nabla u_{k-1}|^2\right)\, dx+3\int_\Omega|\nabla  e^{\psi_k} |^2\, dx-\tau\io\psi_k e^{-\psi_k}dx\leq c.
			\end{equation}
			 Multiplying through the above inequality by $\tau$, summing up the resulting one over $k$, and taking a note of the fact that
			 $$
			 \sbj e^{-\sbj}\leq e^{-1}\ \ \mbox{on the set where $\sbj\geq 0$,}
			 $$ we obtain the proposition. 
			\end{proof}

				Now we are ready to present the next proposition.
				\begin{prop}
					We have
					\begin{equation}
					\max_{0\leq t\leq T}
					\int_\Omega\left(\partial_t\utj\right)^2\, dx
					+\int_{\Omega_T}\left(\partial_t\rtj\right)^2\, dxdt\leq c.
					\end{equation}
					\end{prop}
				\begin{proof}	First we claim
					\begin{equation}
					\int_\Omega\left(\frac{u_1-u_{0j}(x)}{\tau}\right)^2\, dx\leq c.
					\end{equation}
					To see this, we let $k=1$ in (\ref{s31}) to obtain
					\begin{equation}
					\frac{u_1-u_{0j}(x)}{\tau}-\Delta e^{3\psi_1}+\tau\psi_1=0\ \ \mbox{on $\Omega$.}
					\end{equation}
					Multiply through this equation by $\frac{u_1-u_{0j}(x)}{\tau}$ and integrate the resulting equation over $\Omega$ to deduce
				\begin{equation}
				\int_\Omega\left(\frac{u_1-u_{0j}(x)}{\tau}\right)^2\, dx-\int_\Omega\Delta e^{3\psi_1}\frac{u_1-u_{0j}(x)}{\tau}\, dx+\tau\int_\Omega\psi_1\frac{u_1-u_{0j}(x)}{\tau}\, dx=0.\label{ui1}
				\end{equation}
				Note that
				\begin{eqnarray}
				\int_\Omega\Delta e^{3\psi_1}\frac{u_1-u_{0j}(x)}{\tau}\, dx&=&\int_\Omega\Delta\left(e^{3\psi_1}-e^{3\psi_0}\right) \frac{u_1-u_{0j}(x)}{\tau}\, dx\nonumber\\
				&&+\int_\Omega\Delta e^{3\psi_0} \frac{u_1-u_{0j}(x)}{\tau}\, dx\nonumber\\
				&=&\int_\Omega\left(e^{3\psi_1}-e^{3\psi_0}\right) \Delta\frac{u_1-u_{0j}(x)}{\tau}\, dx\nonumber\\
				&&
				+\int_\Omega\Delta \btj \frac{u_1-u_{0j}(x)}{\tau}\, dx\nonumber\\
				&\leq&\frac{1}{\tau}\int_\Omega\left(e^{3\psi_1}-e^{3\psi_0}\right) (e^{-\psi_1}-e^{-\psi_0})\, dx\nonumber\\
				&&+\varepsilon\int_\Omega\left(\frac{u_1-u_{0j}(x)}{\tau}\right)^2\, dx+c(\varepsilon)\nonumber\\
				&\leq&\varepsilon\int_\Omega\left(\frac{u_1-u_{0j}(x)}{\tau}\right)^2\, dx+c(\varepsilon),\ \ \varepsilon>0.\label{ui2}
				\end{eqnarray}
				Here we have used the fact that
				$$\left(e^{3\psi_1}-e^{3\psi_0}\right) (e^{-\psi_1}-e^{-\psi_0})\leq 0.$$
				By Proposition 3.1 and  (\ref{slt}), we have
				\begin{eqnarray}
				\left|\tau\int_\Omega\psi_1\frac{u_1-u_{0j}(x)}{\tau}\, dx\right|&\leq&\varepsilon\int_\Omega\left(\frac{u_1-u_{0j}(x)}{\tau}\right)^2\, dx+c(\varepsilon)\tau^2\int_{\Omega}\left(\psi_1\right)^2dx\nonumber\\
				&\leq&\varepsilon\int_\Omega\left(\frac{u_1-u_{0j}(x)}{\tau}\right)^2\, dx+c(\varepsilon).\label{ui3}
				\end{eqnarray}
				Our claim is an easy consequence of (\ref{ui1}), (\ref{ui2}), and (\ref{ui3}).
				
					For $k=2,3,\cdots, j$, we derive from (\ref{s31}) that
					\begin{equation}
					\frac{1}{\tau}\left(\frac{u_k-u_{k-1}}{\tau}-\frac{u_{k-1}-u_{k-2}}{\tau}\right)-\Delta\left(\frac{\esk-e^{3\psi_{k-1}}}{\tau}\right)+(\psi_k-\psi_{k-1})=0\ \ \mbox{in $\Omega$.}
					\end{equation}
					Multiply through this equation by $\frac{u_k-u_{k-1}}{\tau}$ and integrate the resulting equation over $\Omega$ to obtain
					\begin{eqnarray}
					\lefteqn{\frac{1}{\tau}\int_\Omega
					\left(\frac{u_k-u_{k-1}}{\tau}-\frac{u_{k-1}-u_{k-2}}{\tau}\right)
					\frac{u_k-u_{k-1}}{\tau} \, dx}\nonumber\\
					&&-\int_\Omega\Delta
					\left(\frac{\esk-e^{3\psi_{k-1}}}{\tau}\right)\frac{u_k-u_{k-1}}{\tau} \, dx
					+\int_\Omega \frac{\uk-\uko}{\tau}(\psi_k-\psi_{k-1})\, dx =0.\label{otn4}
					\end{eqnarray}
					Note that 
					$$\left(\esk-e^{3\psi_{k-1}}\right)\mid_{\partial\Omega}=\left(u_k-u_{k-1}\right)\mid_{\partial\Omega}=0.$$
					The second integral on the left-hand side of (\ref{otn4}) can be evaluated from (L4) as follows:
					\begin{eqnarray}
				\lefteqn{-\int_\Omega\Delta\left(\frac{\esk-e^{3\psi_{k-1}}}{\tau}\right)\frac{u_k-u_{k-1}}{\tau} \,dx}\nonumber\\
				&=&
				-\int_\Omega\left(\frac{\esk-e^{3\psi_{k-1}}}{\tau}\right)
				\Delta\left(\frac{u_k-u_{k-1}}{\tau}\right) \, dx\nonumber\\
				&=&-\int_\Omega\left(\frac{\esk-e^{3\psi_{k-1}}}{\tau}\right)
				\left(\frac{e^{-\psi_k}-e^{-\psi_{k-1}}}{\tau}\right) \, dx\nonumber\\
				&\geq&3\int_\Omega\left(\frac{e^{\psi_{k}}-e^{\psi_{k-1}}}{\tau}\right)^2
				 \, dx. \label{otn2}
\end{eqnarray}
Invoking Proposition 3.1 and  (\ref{slt}), we have
\begin{eqnarray}
\left|\int_\Omega \frac{\uk-\uko}{\tau}(\psi_k-\psi_{k-1})\, dx \right|&\leq &\int_\Omega\left(\frac{u_k-u_{k-1}}{\tau}\right)^2\, dx+c\int_\Omega\left(\psi_k-\psi_{k-1}\right)^2
\, dx\nonumber\\
&\leq&\int_\Omega\left(\frac{u_k-u_{k-1}}{\tau}\right)^2\, dx+c.\label{otn5}
\end{eqnarray}
Using (\ref{otn2}) and (\ref{otn5}) in (\ref{otn4}) yields
\begin{eqnarray}
\lefteqn{\frac{1}{2\tau}\int_{\Omega}
\left(\left(\frac{u_k-u_{k-1}}{\tau}\right)^2-\left(\frac{u_{k-1}-u_{k-2}}{\tau}\right)^2\right)\,dx}\nonumber\\
&&+ 3\int_{\Omega}\left(\frac{e^{\psi_k}-e^{\psi_{k-1}}}{\tau}\right)^2\, dx
\leq \int_\Omega\left(\frac{u_k-u_{k-1}}{\tau}\right)^2\, dx+c..
\end{eqnarray}
Multiply through the inequality by $\tau$, sum up the resulting inequality over $k$, and thereby derive
\begin{equation}
\frac{1}{2}\int_{\Omega}
	\left(\partial_t\utj\right)^2\,dx+ 3\int_{\Omega_s}\left(\partial_t\rtj\right)^2\, dx\, dt
\leq \int_{\Omega_s}\left(\partial_t\utj\right)^2\, dx \, dt+c.
\end{equation}
An application of Gronwall's inequality establishes the desired result.
					\end{proof}
					
					In summary, we have:
	\begin{prop}There hold: 
			\begin{enumerate}
				\item[(C1)] the sequence $\{\rbj^3\}$ is bounded in $ L^\infty(0,T;W^{2,2}(\Omega))$;
			\item[(C2)] $\{\utj\}$ and $\{\rtj\}$ are both bounded in $W^{1,2}(\Omega_T)$.
		\end{enumerate}
	\end{prop}
	\begin{proof} (C1) is an easy consequence of the preceding proposition and the equation 
		\begin{equation}
		\Delta \esbj=\frac{\partial\utj}{\partial t}+\tau\sbj\ \ \ \mbox{in $\omt$.}
		\end{equation}

		To see that $\{\rtj\}$ is bounded in $W^{1,2}(\Omega_T)$, we estimate
		\begin{eqnarray*}
			\rtj^2(x,s)-\rho_0^2(x)&=&\int_0^s\frac{\partial}{\partial t}\rtj^2(x,t)\, dt
			=2\int_0^s\rtj(x,t)\frac{\partial}{\partial t}\rtj(x,t)\, dt\\
			&\leq&\int_0^s\rtj^2(x,t)\, dt
			+\int_0^s\left(\frac{\partial}{\partial t}\rtj(x,t)\right)^2\, dt.
		\end{eqnarray*}
		Integrate this inequality over $\Omega$ and then apply Gronwall's inequality to obtain
		\begin{equation}
		\max_{0\leq t\leq T}\int_{\Omega}|\rtj|^2 \, dx\leq c.	
		\end{equation}					
	Next we compute
		\begin{eqnarray*}
			\lefteqn{\int_{\Omega_T}|\nabla\rtj|^2 \, dxdt}\nonumber\\
			 &=&
			\sum_{k=1}^j\int_{t_{k-1}}^{t_k}
			\int_{\Omega} \left|\frac{t-t_k}{\tau}\nabla e^{\psi_k}+(1-\frac{t-t_k}{\tau})\nabla e^{\psi_{k-1}}\right|^2dxdt\\
			&\leq&\sum_{k=1}^j\int_{t_{k-1}}^{t_k}\left[ \frac{t-t_k}{\tau}\int_{\Omega}\left|\nabla e^{\psi_k}\right|^2dx+(1-\frac{t-t_k}{\tau})\int_{\Omega}\left|\nabla e^{\psi_{k-1}}\right|^2dx\right]dt\\
			&=&\sum_{k=1}^j\tau\left[ \int_{\Omega}\left|\nabla e^{\psi_k}\right|^2dx+\int_{\Omega}\left|\nabla e^{\psi_{k-1}}\right|^2dx\right]\\
			&\leq&c\left[ 2\int_{\Omega_T}\left|\nabla e^{\overline{\psi}_j} \right|^2dxdt+\tau\int_{\Omega}\left|\nabla e^{\psi_0}\right|^2dx\right]	\leq c.
		\end{eqnarray*}	
		The last step is due to Proposition 3.2.	
	The proof of the boundedness of $\{\utj\}	$ in $W^{1,2}(\Omega_T)$ is entirely similar. The proof is completed.	
		\end{proof}	
		\begin{prop}We have: 
			\begin{enumerate}
				\item[(R1)] $\{\ubj\}$ and $\{\rbj\}$ are both precompact in $L^2(\Omega_T)$;
				\item[(R2)] the sequence $\{\rbj^3\}$ is precompact in $ L^2(0,T;W^{1,2}(\Omega))$.
					\end{enumerate}
		\end{prop}	
			\begin{proof}
						For $t\in(t_{k-1}, t_k]$, we calculate from (\ref{dut}) and (\ref{dub}) that
						\begin{eqnarray*}
						\utj(x,t)-\ubj(x,t)&=&\frac{t-t_k}{\tau}(\uk-\uko)\\
						&=&(t-t_k)\frac{\partial}{\partial t}\utj(x,t).
						\end{eqnarray*}
						Consequently, we derive, with the aid of Proposition 3.3, that
						\begin{equation}
						\int_{\Omega}\left(\utj(x,t)-\ubj(x,t)\right)^2dx\leq\tau^2 \int_{\Omega}\left(\frac{\partial}{\partial t}\utj(x,t)\right)^2dx\leq c\tau^2, \ \ t\in(t_{k-1}, t_k].\label{uls}
						\end{equation}
						Similarly, we have
								\begin{eqnarray}
							\int_{\Omega_T}|\rtj-\rbj|^2dxdt &=&\sum_{k=1}^{j}\int_{t_{k-1}}^{t_k}(t-t_k)^2\int_{\Omega}\left(\frac{e^{\psi_k}-e^{\psi_{k-1}}}{\tau}\right)^2dxdt
							\nonumber\\
							&=&\sum_{k=1}^{j}\tau^3\int_{\Omega}\left(\frac{\partial}{\partial t}\rtj\right)^2dx\nonumber\\
							&=&\tau^2\int_{\Omega_T}\left(\frac{\partial}{\partial t}\rtj\right)^2dxdt\leq c\tau^2.\label{rls}
							\end{eqnarray}
							This implies (C1).
						
							We can infer from (C1) and (R1) that $\{\rbj^3\}$ is precompact in $ L^2(\Omega_T))$.
							 Then we compute
							\begin{eqnarray}
							\int_{\Omega_T}|\nabla(\rbj^3-\overline{\rho}_i^3)|^2\, dxdt
							&=&\int_{\Omega_T}(\Delta\esbj-\Delta e^{\overline{3\psi}_i})(\rbj^3-\overline{\rho}_i^3)\, dxdt\nonumber\\
							&\leq &c\left(\int_{\Omega_T}(\rbj^3-\overline{\rho}_i^3)^2\, dxdt\right)^{\frac{1}{2}}
							\end{eqnarray}
							for each $i,j$.  This implies the desired result.
							
					\end{proof}

			We are ready to prove Theorem 1.1.
			\begin{proof}[Proof of Theorem 1.1]								
						Passing to subsequences if necessary, we may assume
						\begin{eqnarray}
						\utj &\rightharpoonup & \mbox{$u$ weakly in $W^{1,2}(\Omega_T)$ and strongly in $L^2(\Omega_T)$},\label{utl}\\
						\rtj &\rightharpoonup & \mbox{$\rho$ weakly in $W^{1,2}(\Omega_T)$ and strongly in $L^2(\Omega_T)$}\label{rtl}.
						\end{eqnarray}
						In view of (\ref{uls}) and (\ref{rls}), we immediately have 
							\begin{eqnarray}
						\rbj &\rightharpoonup &\rho \mbox{ weakly in $L^2(0,T;W^{1,2}(\Omega))$ and strongly in $L^2(\Omega_T)$,}\nonumber\\
						&&\mbox{ and therefore a.e. on $\Omega_T$ (pass to a further subsequence if need be)},\label{otn16}\\
							\rbj^3 &\rightarrow &\rho^3 \mbox{ weakly in $L^2(0,T;W^{2,2}(\Omega))$}\nonumber\\
							&&\mbox{ and stronly in $L^2(0,T;W^{1,2}(\Omega))$ },\label{otn17}\\
								\ubj &\rightharpoonup &u \mbox{ weakly in $L^2(0,T;W^{1,2}(\Omega))$, strongly in $L^2(\Omega_T)$,}\nonumber\\
								&&\mbox{ and a.e. on $\Omega_T$}.\label{otn15}
						\end{eqnarray}
						Proposition 3.1 also asserts that
						\begin{equation}
						\sbj\rightharpoonup \psi \ \ \mbox{weakly in $L^p(\Omega_T)$ for each $p>1$.}
						\end{equation}
						Thus we can easily pass to the limit in (\ref{omm1}) to obtain
						\begin{equation}
						\partial_tu=\Delta \rho^3\ \ \mbox{a.e on $\omt$.}\label{osm333}
						\end{equation}
							The key question to our development is how to justify passing to the limit in (\ref{omm2}). The difficulty lies in the fact that we do not have enough controls over the sequence
							$\{\frac{1}{\rbj}\}$. In fact, we only have an
							$L^1$ bound for the sequence (Proposition 3.1).
						With the aid of Fatou's Lemma, we deduce from Proposition 3.1 that
						$$\int_{\Omega_T}\frac{1}{\rho} \, dxdt
						\leq\liminf_{j\rightarrow\infty}\int_{\Omega_T}\frac{1}{\rbj}  \, dxdt\leq c.$$
						Therefore, the set
						$$A_0=\{(x,t)\in\Omega_T:\rho(x,t)=0\}$$
						has Lebesque measure $0$. 
						This combined with (\ref{otn16}) asserts that
							\begin{equation}
						\frac{1}{\rbj} \rightarrow \frac{1}{\rho} \ \ \ \mbox{a.e. on $\Omega_T$}.
							\end{equation}
							 Recall (\ref{omm2}) to obtain
						\begin{equation}
						\Delta\ubj\rightarrow \frac{1}{\rho} \ \ \ \mbox{a.e. on $\Omega_T$}.\label{otn13}
						\end{equation}
						Observe from Proposition 3.2 that 
						\begin{equation}
						\mbox{the sequence
							$\{\Delta\ubj\}$ is  bounded in $L^2(0,T;W^{-1,2}(\Omega))$.}
						\end{equation} Hence we have
							\begin{equation}
							\Delta\ubj\rightharpoonup \Delta u\equiv\mu \ \ \ \mbox{weakly in both $\mathcal{M} (\Omega_T)$ and $L^2\left(0,T;W^{-1,2}(\Omega)\right)$}.\label{ow1}
							\end{equation}
The key issue is: do we have
$$\Delta u=\mu=\frac{1}{\rho}?$$
The following proposition addresses this issue.
\begin{prop}The restriction of $\mu$ to the set $\Omega_T\setminus A_0$ is a function. This function is exactly $\frac{1}{\rho}$. That is, the Lebesgue decomposition of
	$\mu$ with respect to the Lebesgue measure is $\frac{1}{\rho}+\nu_s$, where $\nu_s$ is a measure supported in $A_0$, and we have
	\begin{equation}
	\frac{1}{\rho}=\mu \ \ \mbox{on the set $\Omega_T\setminus A_0$.}
	\end{equation}
\end{prop}
\begin{proof} The proof is almost identical to that in \cite{LX}. For the reader's convenience, we reproduce it here.
	Keep in mind that since $\mu\in L^2\left(0,T;W^{-1,2}(\Omega)\right)$ each function in $L^2\left(0,T;W_0^{1,2}(\Omega)\right)$ is $\mu$-measurable,  and thus it is well-defined except on a set of $\mu$ measure $0$. Furthermore,
	$\int_{0}^{T} \langle \mu, v \rangle \, dt=\int_{\Omega_T} v \, d\mu$ 
	for each $v\in L^2\left(0,T;W_0^{1,2}(\Omega)\right) $, where $\langle\cdot,\cdot\rangle$ denotes the duality pairing between $L^2(0,T; W^{-1,2}(\Omega))$ and $L^2(0,T; W_0^{1,2}(\Omega))$.
	For $\varepsilon>0$ let $\te$ be a smooth function on $\mathbb{R}$ having the properties
	$$\te(s)=\left\{\begin{array}{ll}
	1 &\mbox{if $s\geq 2\varepsilon$,}\\
	0&\mbox{if $s\leq \varepsilon$ \hspace{.5in} and}
	\end{array}\right.$$
	$$
	0\leq\te\leq 1\ \ \mbox{on $\mathbb{R}$.}$$
	Then it is easy to verify from (\ref{otn17}) that we still have
	\begin{equation}
	\te(\rbj^3)\rightarrow\te(\rho^3)\ \ \mbox{strongly in $L^2(0,T;W^{1,2}(\Omega))$.}\label{otn14}
	\end{equation}
	Pick a function $\xi$ from $C^\infty(\mathbb{R}^N\times\mathbb{R})$ with $\xi\mid_{\Sigma_T}=0$. Multiply through (\ref{omm2}) by $\xi\, \te(\rbj^3)$ and integrate the resulting equation over $\Omega$ to obtain
	\begin{equation}
	\int_{\Omega_T}\Delta\ubj\te(\rbj^3)\, \xi \, dxdt
	=\int_{\Omega_T}\frac{1}{\rbj} \, \te(\rbj^3)\, \xi \, dxdt.\label{otn20}
	\end{equation}
	For each fixed $\varepsilon$  the sequence $\{\frac{1}{\rbj} \, \te(\rbj^3)\}$ is bounded in $L^\infty(\Omega_T)$. This, along with (\ref{otn16}), gives
	$$\int_{\Omega_T}\frac{1}{\rbj} \, \te(\rbj^3)\, \xi \, dxdt\rightarrow\int_{\Omega_T}\te(\rho^3)\frac{1}{\rho}\, \xi \, dxdt.$$
	Observe from (\ref{otn14}) and (\ref{ow1}) that
	\begin{equation}
	\int_{\Omega_T}\Delta\ubj \, \te(\rbj^3)\, \xi\, dxdt
	=\int_{0}^{T} \langle \Delta\ubj,\te(\rbj^3)\, \xi \rangle\, dt
	\rightarrow\int_{\Omega_T}\te(\rho^3)\, \xi \, d\mu.
	\end{equation}
	Taking $j\rightarrow\infty$ in (\ref{otn20}) yields
	\begin{equation}
	\int_{\Omega_T}\te(\rho^3)\,\xi \,d\mu=\int_{\Omega_T}\te(\rho^3)\frac{1}{\rho} \, \xi \, dxdt.\label{otn21}
	\end{equation}
Remember that $\rho^3\in L^\infty\left(0,T;W^{2,2}(\Omega)\right)$, and thus it is well-defined except on a set of $\mu$ measure $0$. We can easily conclude from the definition of $\te$ that $\{\te(\rho^3)\}$ converges everywhere on the set where $\rho$ is defined as $\varepsilon\rightarrow 0$. With the aid of the Dominated Convergence Theorem, we can take $\varepsilon\rightarrow 0$ in (\ref{otn21}) to obtain
	$$\int_{\Omega_T\setminus A_0}\, \xi \, d\mu=\int_{\Omega_T\setminus A_0}\frac{1}{\rho} \, \xi \, dxdt.$$
	This is true for every $\xi\in C^\infty(\mathbb{R}^N\times\mathbb{R})$ with $\xi\mid_{\Sigma_T}=0$, which means
	\begin{equation}
	\mu=\frac{1}{\rho}\ \ \ \mbox{on $\Omega_T\setminus A_0$.}
	\end{equation}
The proof is complete.
\end{proof}
With Proposition 3.6, we have concluded the proof of Theorem 1.1.\end{proof}

\begin{proof}[Proof of Theorem 1.2]
  Let $\xi$ be a function in $C^\infty(\mathbb{R}^N\times\mathbb{R})$ with the properties
  $$\xi=0\ \ \mbox{on $\Sigma_T$ and}\ \ \ \xi\geq 0 \ \ \mbox{on $\omt$.}$$
  We multiply each term in (\ref{sfa1}) by $\esk\xi$ and integrate the resulting equation over $\Omega$ to obtain
  \begin{equation}
 \int_{\Omega} \frac{e^{-\psi_k}-e^{-\psi_{k-1}}}{\tau}\esk\xi\ dx-\int_{\Omega}\Delta^2\esk\esk\xi\ dx+\tau\int_{\Omega}\Delta\psi_k\esk\xi\ dx=0.\label{sfa5}
  \end{equation}
  Note that
  \begin{eqnarray}
  \left(e^{-\psi_k}-e^{-\psi_{k-1}}\right)e^{2\psi_k}&\leq&\int_{e^{-\psi_{k-1}}}^{e^{-\psi_k}}s^{-2}ds\nonumber\\
  	&=&-\left(e^{\psi_k}-e^{\psi_{k-1}}\right).\label{sfa3}
  \end{eqnarray}
  Keeping in mind (\ref{sfa2}), we calculate
  \begin{eqnarray}
  \lefteqn{-\int_\Omega\Delta^2\esk\esk\xi\, dx}\nonumber\\
  &=&\int_\Omega\nabla\left(\Delta\esk+\tau\ln\boj\right)\nabla\left(\esk\xi\right)\, dx\nonumber\\
  &&-\tau\int_\Omega\nabla\ln\boj\nabla\left(\esk\xi\right)\, dx\nonumber\\
  &=&-\int_\Omega\left(\Delta\esk+\tau\ln\boj\right)\Delta\left(\esk\xi\right)\, dx\nonumber\\
  &&+\tau\int_\Omega\left(\Delta\ln\boj\right)\esk\xi\, dx.\label{sfa4}
  \end{eqnarray}
  Use (\ref{sfa3}) and (\ref{sfa4}) in (\ref{sfa5}), multiply through the resulting equation by $-\tau$, and sum up over $k$ to obtain
  \begin{eqnarray}
  \lefteqn{ \int_{\Omega_T}\partial_t\rtj\rbj\xi dxdt+\int_{\Omega_T}\left(\Delta\rbj^3\right)^2\xi dx dt}\nonumber\\
&&+\int_{\Omega_T}\Delta\rbj^3\left(2\nabla\rbj^3\nabla\xi+\rbj^3\Delta\xi\right) dxdt\nonumber\\
&\leq&-\tau\int_{\Omega_T}\nabla\overline{\psi_j}\nabla\left(\esbj\xi\right) dxdt-\tau\int_{\Omega_T}\nabla\ln\boj\Delta\left(\esbj\xi\right) dxdt\nonumber\\
&&+\tau\int_{\Omega_T}\Delta\left(\ln\boj\right)\esbj\xi dxdt.\label{sfa7}
 \end{eqnarray}
 We estimate from Proposition 3.1 that
 \begin{eqnarray}
\lefteqn{ -\tau\int_{\Omega_T}\nabla\overline{\psi_j}\nabla\left(\esbj\xi\right) dxdt}\nonumber\\
&\leq&-\tau\int_{\Omega_T}\esbj\nabla\overline{\psi_j}\nabla \xi dxdt\nonumber\\
 &\leq&\sqrt{\tau}\left(\tau\int_{\Omega_T}\esbj|\nabla\overline{\psi_j}|^2 dxdt\right)^{\frac{1}{2}}\left(\int_{\Omega_T}\esbj|\nabla\xi|^2 dxdt\right)^{\frac{1}{2}}\nonumber\\
 &\leq &c\sqrt{\tau}\left(\int_{\Omega_T}\esbj|\nabla\xi|^2 dxdt\right)^{\frac{1}{2}}.
 \end{eqnarray}
 Use this in (\ref{sfa7}) and take  $\tau\rightarrow 0$ (or equivalently,  $j\rightarrow \infty$ ) in the resulting inequality to obtain
 \begin{eqnarray}
 \lefteqn{ \int_{\Omega_T}\partial_t\rho\rho\xi dxdt+\int_{\Omega_T}\left(\Delta\rho^3\right)^2\xi dx dt}\nonumber\\
 &&+\int_{\Omega_T}\Delta\rho^3\left(2\nabla\rho^3\nabla\xi+\rho^3\Delta\xi\right) dxdt\leq 0.\label{sfa8}
 \end{eqnarray}
 Here we have used the fact that the norm in a Banach space is weakly lower semi-continuous, i.e., 
 \begin{equation}
 \liminf_{j\rightarrow \infty}\int_{\Omega_T}\left(\Delta\rbj^3\right)^2\xi dx dt\geq\int_{\Omega_T}\left(\Delta\rho^3\right)^2\xi dx dt.
 \end{equation}
 We conclude (\ref{ssa22}) from (\ref{sfa8}). This completes the proof of Theorem 1.2.
 \end{proof}
 
 If $\{\rbj^3\}$ is precompact in $L^2(0,T; W^{2,2}(\Omega))$, then we can easily show that the equality in (\ref{ssa22}) holds. This does not seem likely because the sequence $\{\rbj^3\}$ is not bounded away from $0$ below uniformly.
 
If the space dimension is $1$, then $\rho^3$ is continuous on $\overline{\omt}$. Unfortunately, our approximate solutions $\rbj^3$ are not continuous in the t-variable. It is natural to ask if we can construct a sequence of approximate solutions that are H\"{o}lder continuous on $\overline{\omt}$. In this regard, one immediately thinks of the sequence $\{\rtj\}$. However, it does not seem possible that we can show the
sequence $\{\rtj^3\}$ is bounded in $W^{1,2}(\omt)\cap L^\infty(0,T; W^{2,2}(\Omega))$.
It turns out that the sequence $\{\ctj\}$ defined by
\begin{equation}
\ctj(x,t)=\frac{t-t_{k-1}}{\tau}\rho_k^3(x)+\frac{t_k-t}{\tau}\rho^3_{k-1},\ \ \ t\in(t_{k-1}, t_k]
\end{equation}
will do the trick. To see this,
we derive from (C1) that
 \begin{equation}
 \|\rbj^3\|_\infty\leq c,\ \ \ \|\partial_x\rbj^3\|_\infty\leq c.\label{osm1}
 \end{equation}
  It immediately follows from (C1) that 
 \begin{equation}
 \mbox{$\{\ctj\}$ is bounded in $L^\infty(0,T:W^{2,2}(\Omega))$.}
 \end{equation}
 For $ t\in(t_{k-1}, t_k]$ we compute
 \begin{equation}
 \partial_t\ctj=\frac{1}{\tau}(\rk^3-\rho_{k-1}^3)=\frac{\rho_k-\rho_{k-1}}{\tau}(\rk^2+\rk\rho_{k-1}+\rho_{k-1}^2),
 \end{equation}
 from whence follows
 \begin{equation}
 |\partial_t\ctj|\leq 3|\partial_t\rtj|\|\rbj\|^2_\infty.
 \end{equation}
 This together with (C2) and (\ref{osm1}) implies that
 \begin{equation}
 \mbox{$\{\ctj\}$ is bounded in $W^{1,2}(\Omega_T)$.}
 \end{equation}
 Now we are in a position to use Lemma 2.2, from which we calculate
  for $t\in(t_{k-1}, t_k]$ that
 \begin{eqnarray}
 \left|\ctj-\rbj^3\right|&=&\left|\frac{t-t_k}{\tau}(\rk^3-\rho_{k-1}^3)\right|\nonumber\\
 &=&\left|\frac{t-t_k}{\tau}(\ctj(x, t_k)-\ctj(x, t_{k-1})\right|\nonumber\\
 &\leq& c\tau^{\frac{1}{4}}.\label{osm222}
 \end{eqnarray}
 Subsequently, we have
 \begin{equation}
 \mbox{ $\{\ctj\}$ converges uniformly to $\rho^3$ on
 	$\overline{\Omega_T}$.}
 \end{equation}

  \bigskip
  
  \noindent{\bf Acknowledgment:} The research of JL was partially supported by
  KI-Net NSF RNMS grant No. 1107291 and NSF grant DMS 1514826.

\end{document}